\documentclass[a4paper,reqno,12pt]{amsart} 
\usepackage{amsmath} 

\usepackage{amssymb}      

\usepackage{yhmath}
\usepackage{mathdots}
\usepackage{MnSymbol}

\usepackage{amsthm}      

\usepackage{tikz,amsmath}
\usetikzlibrary{arrows}

\usepackage{epsfig}      

\usepackage{graphicx}
\usepackage[normalem]{ulem}
\usepackage[all,line,
]{xy} 
\CompileMatrices 

   \newcommand{\Aff}{{\operatorname{Aff}}}
 
   \newcommand{\Hom}{\operatorname{Hom}}

\newcommand{\id}{\operatorname{id}} 
\newcommand{\Aut}{\operatorname{Aut}}

 \newcommand{\supp}{\operatorname{supp}}

\newcommand{\ev}{\operatorname{ev}}

\newcommand{\Sup}{\operatorname{sup}}

   \theoremstyle{plain}
   \newtheorem{thm}{Theorem}[section]
   
   \newtheorem{lemma}[thm]{Lemma}  
   \newtheorem{cor}[thm]{Corollary}
   \theoremstyle{definition}
   
   \newtheorem{defn}[thm]{Definition}
   
   \theoremstyle{remark}
   
   \newtheorem{remark}[thm]{Remark}

\newtheorem{poem}[thm]{Additional properties}
\newtheorem{TOM}[thm]{Property}

\usepackage{mdframed,xcolor}

\definecolor{mybgcolor}{gray}{0.8}
\definecolor{myframecolor}{rgb}{.647,.129,.149}

\mdfdefinestyle{mystyle}{
  usetwoside=false,
  skipabove=0.6em plus 0.8em minus 0.2em,
  skipbelow=0.6em plus 0.8em minus 0.2em,
  innerleftmargin=.25em,
  innerrightmargin=0.25em,
  innertopmargin=0.25em,
  innerbottommargin=0.25em,
  leftmargin=-.75em,
  rightmargin=-0em,
  topline=false,
  rightline=false,
  bottomline=false,
  leftline=false,
  backgroundcolor=mybgcolor,
  splittopskip=0.75em,
  splitbottomskip=0.25em,
  innerleftmargin=0.5em,
  leftline=true,
  linecolor=myframecolor,
  linewidth=0.25em,
}

\newmdenv[style=mystyle]{important}

\usepackage{lipsum}

   \numberwithin{equation}{section}

        \date{\today}

\title[Bundles of KMS state spaces]{The bundle of KMS state spaces for flows on a unital $C^*$-algebra}
\author{George A. Elliott and Klaus Thomsen}

\DeclareMathOperator\coker{coker}

\date{\today}

\address{Department of Mathematics, University of Toronto, Toronto, Ontario,
Canada \ M5S 2E4}
\email{ elliott@math.toronto.edu}

\bigskip

\bigskip

\address{Department of Mathematics, Aarhus University, Ny Munkegade, 8000 Aarhus C, Denmark}

\email{matkt@math.au.dk}

\begin{document}

\maketitle

\section{Introduction}

The collection of KMS state spaces for a flow on a unital $C^*$-algebra can be thought of as a bundle of simplices over the real line. This may seem like a far-fetched analogy to other more established notions of bundles because the fibers of the bundle may be empty or they may all be mutually non-isomorphic, but we will show here that it is in fact a well-behaved and useful concept. Specifically, we use it to show that for any given unital separable infinite-dimensional simple AF algebra $A$ and for any configuration of KMS state spaces which occurs for a flow on a unital separable $C^*$-algebra and has the property that the simplex of $0$-KMS states is affinely homeomorphic to the tracial state space of $A$, there is a flow on $A$ with the same configuration. In particular, it follows that for any given closed subset $F$ of real numbers containing $0$ there are flows on $A$ whose KMS spectrum is $F$. This removes the lower boundedness condition which occurs in a recent work by the second author, \cite{Th3}. 

Since we deal with unital AF algebras, there are always $0$-KMS states present. For flows on infinite $C^*$-algebras this is not the case, and in a joint work with Y. Sato we have shown that for any unital, nuclear, purely infinite, simple, separable $C^*$-algebra $A$ in the UCT class and with torsion-free $K_1$ group, and for any configuration of KMS state spaces which occurs for a flow on a unital separable $C^*$-algebra without trace states, there is also a flow on $A$ with the same configuration; see \cite{EST}. In both cases we depend on results from the classification of simple $C^*$-algebras.

While the work in \cite{Th3} was based on ideas from \cite{BEH} and \cite{BEK1}, in the present paper the underlying ideas are closer to those presented by Bratteli, Elliott and Kishimoto in \cite{BEK2}. In particular, the idea of considering the configuration of KMS simplices as a bundle originates from \cite{BEK2}.

\emph{Acknowledgements} The work of the first named author was supported by a Natural Sciences and Engineering Research Council of Canada Discovery Grant. The work of the second named author was supported by the DFF-Research Project 2 `Automorphisms and Invariants of Operator Algebras', no. 7014-00145B.

\section{Proper simplex bundles}

  Let $S$ be a second countable locally compact Hausdorff space and $\pi : S \to \mathbb R$ a continuous map. If the inverse image $\pi^{-1}(t)$, equipped with the relative topology inherited from $S$, is homeomorphic to a compact metrizable Choquet simplex for all $t \in \mathbb R$ we say that $(S,\pi)$ is a \emph{simplex bundle}. We emphasize that $\pi$ need not be surjective, and we consider therefore also the empty set as a simplex. When $(S,\pi)$ is a simplex bundle we denote by $\mathcal A(S,\pi)$ the set of continuous functions $f : S \to \mathbb R$ with the property that the restriction $f|_{\pi^{-1}(t)}$ of $f$ to $\pi^{-1}(t)$ is affine for all $t \in \mathbb R$.

\begin{defn}\label{25-08-21} (Compare \cite{BEK2}.) A simplex bundle $(S,\pi)$ is a \emph{proper simplex bundle} when
\begin{itemize}
\item[(1)] $\pi$ is proper; that is $\pi^{-1}(K)$ is compact in $S$ when $K \subseteq \mathbb R$ is compact, and
\item[(2)] $\mathcal A(S,\pi)$ separate points on $S$; that is for all $x\neq y$ in $S$ there is an $f \in\mathcal A(S,\pi)$ such that $f(x) \neq f(y)$.
\end{itemize}
\end{defn}

Two proper simplex bundles $(S,\pi)$ and $(S',\pi')$ are \emph{isomorphic} when there is a homeomorphism $\phi : S \to S'$ such that $\pi' \circ \phi = \pi$ and $\phi: \pi^{-1}(\beta) \to {\pi'}^{-1}(\beta)$ is affine for all $\beta \in \mathbb R$.

\subsection{Proper simplex bundles from flows}\label{flows}
In this paper all $C^*$-algebras are assumed to be separable and all traces and weights on a $C^*$-algebra are required to be non-zero, densely defined and lower semi-continuous. Let $A$ be a $C^*$-algebra and $\theta$ a flow on $A$. Let $\beta \in \mathbb R$. A $\beta$-KMS weight for $\theta$ is a weight $\omega$ on $A$ such that $\omega \circ \theta_t = \omega$ for all $t$, and 
\begin{equation}\label{27-10-20c}
\omega(a^*a) \ = \ \omega\left(\theta_{-\frac{i\beta}{2}}(a) \theta_{-\frac{i\beta}{2}}(a)^*\right) \ \  \ \forall a \in D(\theta_{-\frac{i\beta}{2}}) \ .
\end{equation}
In particular, a $0$-KMS weight for $\theta$ is a $\theta$-invariant trace. 
 A bounded $\beta$-KMS weight is called a $\beta$-KMS functional and a $\beta$-KMS state when it is of norm one. For states alternative formulations of the KMS condition can be found in \cite{BR}.

 Assume that $A$ is unital. For each $\beta \in \mathbb R$ let $S^\theta_\beta$ be the (possibly empty) set of $\beta$-KMS states for $\theta$. Let $E(A)$ be the state space of $A$, a compact convex set in the weak* topology.
Set
$$
S^\theta = \left\{(\omega, \beta) \in E(A) \times \mathbb R: \ \omega \in S^\theta_\beta \right\} \ ,
$$
and equip $S^\theta$ with the relative topology inherited from the product topology of $E(A) \times \mathbb R$.  Since $S^\theta$ is a closed subset of $E(A)\times \mathbb R$ by Proposition 5.3.23 of \cite{BR}, it follows that $S^\theta$ is a second countable locally compact Hausdorff space. Denote by $\pi^\theta : S^\theta \to \mathbb R$ the projection to the second coordinate. Since the inverse image ${\pi^{\theta}}^{-1}(\beta)$ is homeomorphic to $S^\theta_\beta$, which is a Choquet simplex by Theorem 5.3.30 of \cite{BR}, the pair $(S^\theta,\pi^\theta)$ is a simplex bundle. An obvious application of Proposition 5.3.23 of \cite{BR}, using the compactness of $E(A)$, shows that $\pi^\theta$ is proper. Note that every self-adjoint element $a \in A$ gives rise to an element $\hat{a} \in \mathcal A(S^\theta,\pi^\theta)$ such that $\hat{a}(\omega,\beta) = \omega(a)$. A continuous function $f :  \mathbb R \to \mathbb R$ also gives rise to an element of $\mathcal A(S^\theta,\pi^\theta)$: 
$$
S^\theta \ni (\omega,\beta) \ \mapsto \ f(\beta) .
$$
It follows that $\mathcal A(S^\theta,\pi^\theta)$ separates the points of $S^\theta$, showing that $(S^\theta,\pi^\theta)$ is a proper simplex bundle, which we shall call the \emph{KMS bundle} of the flow. In general, ${\pi^{\theta}}^{-1}(0)$ is the set of $\theta$-invariant trace states and hence non-empty if and only if $A$ has trace states. When $A$ is AF it is the simplex of all trace states of $A$.

\begin{remark}\label{07-09-21f} Let $(S,\pi)$ be a proper simplex bundle. Denote by $\mathcal A_\mathbb R(S,\pi)$ the subset of $\mathcal A(S,\pi)$ consisting of the elements that have a limit at infinity. This is a separable real Banach space (in the supremum norm) containing the constant function $1$, and its state space
\begin{align*}
& E(\mathcal A_\mathbb R(S,\pi)):= \\
&  \left\{ \omega \in \mathcal A_\mathbb R(S,\pi)^* : \ |\omega(f)| \leq \sup_{x \in S}|f(x)| \ \ \forall f \in \mathcal A_\mathbb R(S,\pi) , \ \omega(1) = 1\right\}
\end{align*}
is a metrizable compact convex set in the weak* topology. For $x \in S$, let $\ev_x \in E(\mathcal A_\mathbb R(S,\pi))$ denote evaluation at $x$. For each $\beta \in \mathbb R$, the set
$$
K_\beta := \left\{ \ev_x : \ x \in \pi^{-1}(\beta) \right\}
$$
is a closed convex subset of $E(\mathcal A_\mathbb R(S,\pi))$. With $K = E(\mathcal A_\mathbb R(S,\pi))$, the system $K_\beta, \beta \in \mathbb R$, has the properties required in Theorem 2.1 of \cite{BEK1}. Thus, there is a unital, simple, separable, nuclear $C^*$-algebra $A$ equipped with a $2 \pi$-periodic flow $\theta$ such that $\pi^{-1}(\beta)$ is affinely homeomorphic to the simplex of $\beta$-KMS states for all $\beta \in \mathbb R$. In the construction of $A$ in \cite{BEK1} the algebra appears to depend on $(S,\pi)$ and on the many choices made in the process of its construction, but it was shown in \cite{EST}, based on the Kirchberg-Phillips classification result, that when $\pi^{-1}(0)$ is empty one can take $A$ to be any given separable, simple, nuclear, purely infinite $C^*$-algebra in the UCT class and with torsion free $K_1$ group. It follows from the main result we describe next that when $\pi^{-1}(0)$ is not empty one can take $A$ to be any infinite dimensional unital simple AF algebra whose tracial state space is affinely homeomorphic to $\pi^{-1}(0)$.
\end{remark}

\section{The main result and applications}

\begin{thm}\label{26-08-21} Let $(S,\pi)$ be a proper simplex bundle and let $A$ be a unital infinite-dimensional simple $AF$ algebra with a tracial state space affinely homeomorphic to $\pi^{-1}(0)$. There is a $2\pi$-periodic flow on $A$ whose KMS bundle is isomorphic to $(S,\pi)$.
\end{thm}

By definition the {KMS spectrum} of a flow $\theta$ on a unital $C^*$-algebra is the set of real numbers $\beta$ for which $\theta$ has a $\beta$-KMS state. When the KMS bundle of $\theta$ is isomorphic to a proper simplex bundle $(S,\pi)$, the KMS spectrum of $\theta$ is the range $\pi(S)$ of $\pi$.

To exhibit possible ways to work with proper simplex bundles and to illustrate how Theorem \ref{26-08-21} can be applied let us use it to prove the following statement.

\begin{cor}\label{08-09-21} Let $A$ be a unital infinite-dimensional simple AF algebra and let $F$ be a closed subset of real numbers containing $0$.
\begin{itemize}
\item There is a $2 \pi$-periodic flow on $A$ whose KMS spectrum is $F$ and such that there is a unique $\beta$-KMS state for all $\beta \in F \backslash\{0\}$.
\item There is a $2\pi$-periodic flow $\theta$ on $A$ whose KMS spectrum is $F$ and such that $S^\theta_\beta$ is not affinely homeomorphic to $S^\theta_{\beta'}$ when $\beta,\beta' \in F \backslash \{0\}$ and $\beta \neq \beta'$.
\end{itemize}
\end{cor} 
\begin{proof} Given a proper simplex bundle $(S,\pi)$, set $S_F = \pi^{-1}(F)$ and denote by $\pi_F$ the restriction of $\pi$ to $\pi^{-1}(F)$. The pair $(S_F,\pi_F)$ is again a proper simplex bundle.

For the first item, let $T(A)$ be the tracial state space of $A$ and fix an element $\omega_0 \in T(A)$. Let $S$ be the subset
$$
(T(A) \times \{0\}) \cup \left\{(\omega_0,t) : \ t \in \mathbb R \backslash \{0\} \right\}
$$
of the topological product $T(A) \times \mathbb R$. Let $\pi : S \to \mathbb R$ be the canonical projection. Then $(S,\pi)$ is a proper simplex bundle such that $\pi^{-1}(\beta) = \{(\omega_0,\beta)\}$ when $\beta \neq 0$ and such that $\pi^{-1}(0)$ is a copy of $T(A)$. The existence of the desired flow follows from Theorem \ref{26-08-21} by applying it to the bundle $(S_F,\pi_F)$.

For the second item, note that the KMS bundle of the flow described in Theorem 1.1 of \cite{Th2} is a proper simplex bundle $(S',\pi')$ such that $\pi'^{-1}(0)$ contains only one point, $\pi'(S) = \mathbb R$ and such that $\pi'^{-1}(\beta)$ is not affinely homeomorphic to $\pi'^{-1}(\beta')$ when $\beta \neq \beta'$. With $(S,\pi)$ the bundle defined above set
$$
S'' = \left\{ (x,y) \in S'\times S: \ \pi'(x) = \pi(y) \right\} \ .
$$
Define $\pi'' : S'' \to \mathbb R$ by $\pi''(x,y) = \pi(y)$. Then $(S'',\pi'')$ is a proper simplex bundle and an application of Theorem \ref{26-08-21}, this time to $(S''_F,\pi''_F)$, gives the desired flow.
\end{proof}

\section{Proof of the main result}

\subsection{Tools}

The following lemma is an immediate consequence of Lemma 3.1 of \cite{Th3} and Theorem 2.4 of \cite{Th1}.

\begin{lemma}\label{06-08-21} Let $D$ be a $C^*$-algebra. Denote by $\rho \in \Aut(D)$ an automorphism of $D$ and let $q \in D$ a projection in $D$ which is full\footnote{Recall that a projection $q$ in a $C^*$-algebra $A$ is full when $\overline{AqA} = A$. This is automatic when $q \neq 0 $ and $A$ is simple.
} in $D\rtimes_\rho \mathbb Z$. Let $\hat{\rho}$ denote the restriction to $q(D \rtimes_\rho \mathbb Z)q$ of the dual action on $D\rtimes_\rho \mathbb Z$ considered as a $2\pi$-periodic flow. Let $P : D \rtimes_\rho \mathbb Z \to D$ denote the canonical conditional expectation. For each $\beta \in \mathbb R$, the map $\tau \mapsto \tau \circ P|_{q(D \rtimes_\rho \mathbb Z)q}$ is an affine homeomorphism from the set of traces $\tau$ on $D$ that satisfy 
\begin{equation}\label{11-10-21}
\tau \circ \rho = e^{-\beta} \tau \ \text{and} \ \tau(q) =1,
\end{equation}
onto the simplex of $\beta$-KMS states for $\hat{\rho}$.\end{lemma}

In this lemma the topology on the set of traces of $D$ with the properties \eqref{11-10-21} is given by pointwise convergence on elements from the corner $qDq$ of $D$.

When $D$ is an AF algebra it is well-known that the set of its traces can be identified, via the map $\tau \mapsto \tau_*$, with the set $\Hom^+(K_0(D),\mathbb R)$ of non-zero positive homomorphisms $\phi :  K_0(D) \to \mathbb R$; a fact stated as Lemma 3.5 in \cite{Th3}.
In the setting of Lemma \ref{06-08-21} this implies that when $D$ is an AF algebra the KMS spectrum and the structure of the KMS states for $\hat{\rho}$ can be determined directly from the pair $(K_0(D),\rho_*)$. To see how, we note that by Remark 3.3 in \cite{LN} every $\beta$-KMS state $\tau$ for $\hat{\rho}$ on $q(D \rtimes_\rho\mathbb Z)q$ extends uniquely to a $\beta$-KMS weight $\hat{\tau}$ for the dual action. Since $D$ is the fixed point algebra for the dual action the restriction of $\hat{\tau}$ to $D$ is a trace on $D$, yielding a map
\begin{equation}\label{08-08-21b}
\tau \mapsto \left( \hat{\tau}|_D\right)_* 
\end{equation}
from the set of $\beta$-KMS states $\tau$ for $\hat{\rho}$ to $\Hom^+(K_0(D),\mathbb R)$. Therefore Lemma \ref{06-08-21} has the following consequence.

\begin{cor}\label{06-08-21a} In the setting of Lemma \ref{06-08-21} assume that $D$ is an AF algebra. For each $\beta \in \mathbb R$ the map \eqref{08-08-21b} is an affine homeomorphism from the set of $\beta$-KMS states for $\hat{\rho}$ on $q(D \rtimes_\gamma\mathbb Z)q$ onto the set of positive homomorphisms $\phi \in \Hom^+(K_0(D),\mathbb R)$ that satisfy 
\begin{equation}\label{11-10-21a}
\phi \circ \rho_* = e^{-\beta} \phi \ \text{and} \ \phi([q]) =1.
\end{equation}
\end{cor}

Here the topology on the elements from $\Hom^+(K_0(D),\mathbb R)$ with the properties \eqref{11-10-21a} is given by pointwise convergence on $\{x \in K_0(D): \ 0 \leq x \leq [q]\}$.

Corollary \ref{06-08-21a} will be complemented by the following lemma which helps to control the Elliott invariant of $q(D \rtimes_\gamma\mathbb Z)q$, and to ensure that it is classified by it. It follows from Lemma 3.4 of \cite{Th3}, which is based on arguments from \cite{Sa}, \cite{MS1} and \cite{MS2}.

\begin{lemma}\label{08-08-21a} Let $D$ be a stable AF algebra such that $K_0(D)$ has large denominators.\footnote{An ordered group $(G,G^+)$ has large denominators when the following condition holds: For any $a \in G^+$ and any $n \in \mathbb N$ there are an element $b \in G$ and an $m \in \mathbb N$ such that $ nb\leq a \leq mb$; see \cite{Ni}. } For any order automorphism $\alpha \in \Aut(K_0(D))$ of $K_0(D)$ there is an automorphism $\gamma \in \Aut(D)$ of $D$ such that\begin{enumerate}
\item[(a)] $\gamma_* = \alpha$ on $K_0(D)$,
\item[(b)] the restriction map 
$\mu \ \mapsto \ \mu|_D$
is a bijection from traces $\mu $ on $D \rtimes_{\gamma} \mathbb Z$ onto the $\gamma$-invariant traces on $D$,  and
\item[(c)] $D \rtimes_{\gamma} \mathbb Z$ is $\mathcal Z$-stable; that is, $(D \rtimes_{\gamma} \mathbb Z)\otimes \mathcal Z \simeq D \rtimes_{\gamma} \mathbb Z$ where $\mathcal Z$ denotes the Jiang-Su algebra, \cite{JS}. 
\end{enumerate} 
\end{lemma}

Given a proper simplex bundle $(S,\pi)$ and a closed subset $F \subseteq \mathbb R$ we denote by $(S_F,\pi_F)$ the proper simplex bundle with $S_F = \pi^{-1}(F)$ and $\pi_F$ is the restriction of $\pi$ to $S_F$. The following lemma relates $\mathcal A(S_F,\pi_F)$ to $\mathcal A(S,\pi)$ and it will be a crucial tool in the following.

\begin{lemma}\label{03-09-21a} Let $(S,\pi)$ be a proper simplex bundle and $F \subseteq \mathbb R$ a closed subset. 
\begin{itemize}
\item[(1)] The map $\mathcal A(S,\pi) \to \mathcal A(S_F,\pi_F)$ given by restriction is surjective.
\item[(2)] Let $f_1,f_2,g_1,g_2 \in \mathcal A(S,\pi)$ such that $f_i(x) < g_j(x)$ for all $x \in S$ and all $i,j \in \{1,2\}$. Assume that there is an element $h^F \in \mathcal A(S_F,\pi_F)$ such that 
$$
f_i(x) < h^F(x) <  g_j(x) \ \ \forall x \in S_F, \ \forall i,j \in \{1,2\} .
$$ 
There is an element $h \in \mathcal A(S,\pi)$ such that $h(y) = h^F(y)$ for all $y \in S_F$ and 
$$
f_i(x) < h(x) < g_j(x) \ \ \forall x \in S, \ \forall i,j \in \{1,2\} .
$$
\end{itemize}
\end{lemma}
\begin{proof} (1) Let $h \in \mathcal A(S_F,\pi_F)$. For each $n \in \mathbb N$ the pair $(S_{[-n,n]},\pi_{[-n,n]})$ is a compact simplex bundle in the sense of \cite{BEK2} and it follows from Lemma 2.2 in \cite{BEK2} that there are elements $f_n \in \mathcal A(S_{[-n,n]},\pi_{[-n,n]})$ such that the restriction $f_n|_{S_{F \cap [-n,n]}}$ of $f_n$ to $S_{F \cap [-n,n]}$ agrees with $h|_{S_{F \cap [-n,n]}}$. For $n \in \mathbb N$ choose a continuous function $\chi_n : \mathbb R \to [0,1]$ such that $\chi_n (t) = 1$ for $t \leq n-\frac{1}{2}$ and $\chi_n(t) = 0$ for $t \geq n$. Define $f'_n : S_{[-n,n]} \to \mathbb R$ recursively by
$$
f'_1(x) = (1-\chi_1(|\pi(x)|))f_2(x) + \chi_1(|\pi(x)|)f_1(x) \ ,
$$
and then $f'_n$ for $n \geq 2$ such that $f'_{n}(x) = f'_{n-1}(x)$ when $x \in \pi^{-1}([-n+1,n-1])$ and $f'_n(x) =  (1-\chi_n(|\pi(x)|))f_{n+1}(x) + \chi_n(|\pi(x)|)f_n(x)$ when $x \in \pi^{-1}([n-1,n] \cup [-n,-n+1])$. Then $f'_n|_{S_{F \cap [-n,n]}} = h|_{S_{F \cap [-n,n]}}$ and since $f'_{n+1}$ extends $f'_n$ for all $n$, there is an element $f \in \mathcal A(S,\pi)$ such that $f|_{S_{[-n,n]}} = f'_n$. This element $f$ extends $h$.

The assertion (2) follows in a similar way on using Lemma 2.3 in \cite{BEK2}.
\end{proof}

\subsection{Dimension groups from proper simplex bundles}\label{I}

Given a Choquet simplex $\Delta$, as usual denote by $\Aff \Delta$ the set of real-valued continuous and affine functions on $\Delta$. Fix a proper simplex bundle $(S,\pi)$ with $\pi^{-1}(0)$ non-empty. Let $H$ be a torsion free abelian group and $\theta : H \to \Aff \pi^{-1}(0)$ a homomorphism. We assume
\begin{itemize}
\item[(1)] there is an element $u \in H$ such that $\theta(u) = 1$, and
\item[(2)] $\theta(H)$ is dense in $\Aff \pi^{-1}(0)$.
\end{itemize}
Set 
$$
H^+ = \left\{h \in H: \ \theta(h)(x) > 0 \ \ \forall x \in \pi^{-1}(0) \right\} \cup \{0\} \ .
$$
Then $(H,H^+)$ is a simple dimension group; see \cite{EHS}.

It follows from (1) of Lemma \ref{03-09-21a} that the map $r : \mathcal A(S,\pi) \to \Aff \pi^{-1}(0)$ given by restriction is surjective and we can therefore choose a linear map $L :  \Aff \pi^{-1}(0) \to \mathcal A(\pi,S)$ such that $r\circ L = \id$. We arrange, as we can, that $L(1) = 1$. Define $\hat{L} : \bigoplus_{\mathbb Z} H \to \mathcal A(S,\pi)$ by
$$
\hat{L}\left((h_n)_{n \in \mathbb Z}\right)(x) = \sum_{n \in \mathbb Z} L(\theta(h_n))(x)e^{n\pi(x)} \ .
$$ 
Let $\mathbb Q[e^{-\pi}, 1-e^{-\pi}]$ denote the $\mathbb Q$-linear span of the functions defined on $S \backslash \pi^{-1}(0)$ by
\begin{equation}\label{07-09-21a}
 x \mapsto e^{n\pi(x)}(1-e^{-\pi(x)})^l
 \end{equation}
for some $n, l \in \mathbb Z$. For each $k \in \mathbb N$, choose continuous functions $\psi^0_k, \psi^\pm_k : \mathbb R \to [0,1]$ such that 
\begin{itemize}
\item[{}] $\psi^0_k(t) = 1, \ -\frac{1}{2k} \leq t \leq \frac{1}{2k}$, 
\item[{}] $\psi^-_k(t) = 1, \ t \leq - \frac{1}{k}$,
\item[{}] $\psi^+_k(t) = 1, \ t \geq \frac{1}{k}$, and
\item[{}] $\psi_k^-(t) + \psi_k^0(t) + \psi_k^+(t) = 1$ for all $t \in \mathbb R$. 
\end{itemize}
Consider the countable subgroup of $ \mathcal A(S,\pi)$
$$
G_k :=  \mathbb Q[e^{-\pi}, 1-e^{-\pi}]\psi^-_k \circ \pi + \hat{L}(\bigoplus_{\mathbb Z} H ) \psi^0_k \circ \pi  + \mathbb Q[e^{-\pi}, 1-e^{-\pi}]\psi^+_k \circ \pi \ .
$$
In what follows we denote the support of a real-valued function $f$ by $\supp f$. Let $\mathcal A_{00}(S,\pi)$ denote the set of elements $f$ from $\mathcal A(S,\pi)$ for which $\supp f$ is compact and contained in $S \backslash \pi^{-1}(0)$. Since the topology of $S$ is second countable we can choose a countable subgroup $G_{00}$ of $\mathcal A_{00}(S,\pi)$ with the following density property:

\begin{TOM}\label{07-09-21} For all $N \in \mathbb N$, all $\epsilon > 0$ and all $f \in \mathcal A_{00}(S,\pi)$ with 
$\supp f \subseteq \pi^{-1}(]-N,N[ \backslash \{0\})$, there is $g \in G_{00}$ such that 
$$\sup_{x \in S} |f(x) -g(x)| < \epsilon
$$ 
and $\supp g \subseteq \pi^{-1}(]-N,N[ \backslash \{0\})$.
\end{TOM}

When $f_1 \in \hat{L}(\bigoplus_{\mathbb Z} H ) $ and $f^\pm_2 \in  \mathbb Q[e^{-\pi}, 1-e^{-\pi}]$ the difference between 
$$
f^-_2 \psi^-_k \circ \pi + f_1\psi^0_k \circ \pi  + f^+_2\psi^+_k \circ \pi
$$ 
and 
$$
f^-_2 \psi^-_{k+1} \circ \pi + f_1\psi^0_{k+1} \circ \pi  + f_2^+\psi^+_{k+1} \circ \pi
$$ 
is an element of $\mathcal A_{00}(S,\pi)$, and so by enlarging $G_{00}$ we can ensure that
\begin{equation}\label{01-09-21c}
G_k + G_{00} \subseteq G_{k+1} + G_{00} \ .
\end{equation}
Let $\alpha_0 \in \Aut \mathcal A(S,\pi)$ be defined by
$$
\alpha_0(f)(x) = e^{-\pi(x)}f(x) \ .
$$
Since $\alpha_0(\mathcal A_{00}) =\mathcal A_{00}$ and since 
$(1-e^{-\pi})^{-1}\mathcal A_{00}(S,\pi) \subseteq \mathcal A_{00}(S,\pi)$,
we can enlarge $G_{00}$ further to achieve that 
\begin{equation}\label{07-09-21d}
\alpha_0(G_{00}) = G_{00}
\end{equation}
and that 
\begin{equation}\label{07-09-21e}
(\id - \alpha_0)(G_{00}) = G_{00} .
\end{equation}
We define
$$
G = \bigcup_{k=1}^\infty (G_k + G_{00}) .
$$
Let $\sigma \in \Aut \left( \bigoplus_{n \in \mathbb Z} H\right)$ denote the shift:
$$
\sigma\left((h_n)_{n \in \mathbb Z}\right) = \left( h_{n+1}\right)_{n \in \mathbb Z} \ .
$$
Then $\alpha_0 \circ \hat{L} = \hat{L} \circ \sigma$, which implies that $\alpha_0(G_k) = G_k$ and hence 
$$\alpha_0 (G) = G .
$$ 
Set
$$
\mathcal A(S,\pi)^+ = \left\{ f \in \mathcal A(S,\pi): \ f(x) > 0 \ \forall x \in S \right\} \cup \{0\} \ 
$$
and
$$
G^+ = G \cap \mathcal A(S,\pi)^+ \ .
$$ 

\begin{lemma}\label{01-09-21} The pair $(G,G^+)$ has the following properties.
\begin{itemize}
\item[(1)] $G^+ \cap (-G^+) = \{0\}$.
\item[(2)] $G = G^+ - G^+$.
\item[(3)] $(G,G^+)$ is unperforated, i.e., $n \in \mathbb N \backslash \{0\}, \ g \in G, \ ng \in G^+ \Rightarrow g \in G^+$. \item[(4)] $(G,G^+)$ has the strong Riesz interpolation property, i.e. if $f_1,f_2,g_1,g_2 \in G$ and $f_i < g_j$ in $G$ for all $i,j \in \{1,2\}$, then there is an element $h \in G$ such that
$$
f_i < h < g_j
$$
for all $i,j \in \{1,2\}$. 
\end{itemize}
\end{lemma}
\begin{proof} (1) and (3) are obvious. (2): Let $f \in G$. Then $f \in G_k+ G_{00}$ for some $k \in \mathbb N$ and we can write
$$
f = h^-\psi^-_k\circ \pi + f_0\psi^0_k\circ \pi + h^+\psi^0_k \circ \pi + g
$$
where $h^\pm \in \mathbb Q[e^{-\pi}, 1-e^{-\pi}], \ f_0 \in \hat{L}(\bigoplus_{\mathbb Z} H)$ and $g \in G_{00}$. By definition of $\mathbb Q[e^{-\pi}, 1-e^{-\pi}]$ there are $n,m \in \mathbb N$ and $M > 0$ such that
$$
h^-(x) \leq e^{-n\pi(x)} \ \ \forall x \in \pi^{-1}(]-\infty, -M]) \ ,
$$
and
$$
h^+(x) \leq e^{m\pi(x)} \ \ \forall x \in \pi^{-1}([M,\infty[) \ .
$$
Since $f_0 \psi^0_k \circ \pi$ and $g$ are compactly supported there is $K \in \mathbb N$ such that $f(x) <g(x)$ for all $x \in S$, where
$$
g = K\left(e^{-n\pi} \psi^-_k \circ \pi + \psi^0_k\circ \pi + e^{m\pi}\psi^+_k\circ \pi \right) \ \in \ G^+ \ .
$$
Note that $f =  g -(g-f) \in G^+- G^+$. 

(4): Since $\theta(H)$ has the Riesz interpolation property for the strict order by Lemma 3.1 in \cite{EHS}, there is $h_0 \in \theta(H)$ such that $f_i(x) < h_0(x) < g_j(x)$ for all $i,j$ and all $x \in \pi^{-1}(0)$. We claim that there are elements $h^{\pm} \in \mathbb Q[e^{-\pi}, 1-e^{-\pi}]$ and $K^\pm \in \mathbb N$ such that $f_i(x) < h^-(x) < g_j(x)$ when $\pi(x) \leq -K^-$ and $f_i(x) < h^+(x) < g_j(x)$ when $\pi(x) \geq  K^+$. To find $h^+$ and $K^+$, note that we may assume that $\pi(S)$ contains arbitrarily large positive numbers; otherwise we take $K^+$ larger that $\Sup \pi(S)$ and $h^+ = 0$. By definition of $G$ there is $N \in \mathbb N$ so large that there are polynomials $p_1,p_2,q_1,q_2$ with rational coefficients such that
$$
\left(e^{\pi(x)}(e^{\pi(x)} -1)\right)^N f_i(x)  =  p_i(e^{\pi(x)})
$$
and
$$
\left(e^{\pi(x)}(e^{\pi(x)} -1)\right)^N g_j(x)  =  q_j(e^{\pi(x)})
$$
for all $i,j$ and all large $x$. Then $p_i(y) < q_j(y)$ for all large elements $y$ of $e^{\pi(S)}$ and it follows therefore from Lemma \ref{04-09-21} that there is a polynomial $h'$ with rational coefficients such that $p_i(x) < h'(x) <q_j(x)$ for all $i,j$ and all large $x$. Set
$$
h^+(x) = \left(e^{\pi(x)}(e^{\pi(x)} -1)\right)^{-N} h'(e^{\pi(x)}) \ .
$$
Then $h^+ \in \mathbb Q[e^{-\pi}, 1-e^{-\pi}]$ and if $K^+$ is large enough $f_i(x) < h^+(x) < g_j(x)$ for all $i,j$ and all $x \in \pi^{-1}([K^+,\infty))$. The pair $h^-,K^-$ is constructed in a similar way: Multiply each of the functions from $\{f_1,f_2,g_1,g_2\}$ with the same function of the form
$$
(e^{-\pi})^{N}( e^{-\pi}-1)^{N}
$$
to get them into the form $x \mapsto p(e^{-\pi(x)})$ for $p$ a polynomial with rational coefficients and apply Lemma \ref{04-09-21}.

 Having $h^\pm$ and $K^\pm$ we use the statement (2) of Lemma \ref{03-09-21a} to find $H\in \mathcal A(S,\pi)$ such that $H(x) = h^-(x)$ when $\pi(x) \leq -K^-$, $ H(x) = h^+(x)$ when $\pi(x) \geq K^+$, $H(x) = h_0(x)$ when $x \in \pi^{-1}(0)$, and
$
f_i(x) < H(x) < g_j(x)$
for all $x \in S$ and all $i,j$. Set
$$
H'(x) = H\psi^-_k \circ \pi  + L(h_0) \psi^0_k\circ \pi  + H \psi^+_k\circ \pi   .
$$
If $k$ is large enough we have 
$$
f_i(x) < H'(x) < g_j(x)
$$
for all $x \in S$ and all $i,j$. Set 
$$
H''(x) = H'(x) - L(h_0)\psi^0_k \circ \pi  - h^-\psi^-_k\circ \pi  - h^+\psi^+_k\circ \pi   ,
$$
and note that $\sup H'' \subseteq ]-K^-,K^+[ \backslash \{0\}$. Set $K = \max \{K^-,K^+\}$ and let $\delta > 0$ be given, smaller than $g_j(x) - H'(x)$ and $H'(x) -f_i(x)$ for all $i,j$ and all $x \in \pi^{-1}([-K,K])$.
By Property \ref{07-09-21} there is an element $g' \in G_{00}$ such that 
$$
\supp g'  \subseteq \pi^{-1}(]-K, K[ \backslash \{0\}) \ 
$$
and $\sup_{x \in S}\left|g'(x) - H''(x)\right| < \frac{\delta}{2}$. Then 
$$
h = g' +  L(h_0)\psi^0_k \circ \pi + h^-\psi^-_k\circ \pi + h^+\psi^+_k\circ \pi \ \in \ G
$$ and $f_i(x) < h(x) < g_j(x)$ for all $i,j$ and all $x \in S$.
\end{proof}

\begin{lemma}\label{04-09-21} Let $p_i,q_j, \ i,j \in \{1,2\}$, be polynomials with rational coefficients. Assume that there is a sequence $\{x_n\}$ in $\mathbb R$ such that $\lim_{n \to \infty} x_n = \infty$ and such that $p_i(x_n) < q_j(x_n)$ for all $i,j,n$. It follows that there is a polynomial $h$ with rational coefficients and a $K > 0$ such that
$$
p_i(x) < h(x) < q_j(x)
$$
for all $i,j \in \{1,2\}$ and all $x \geq K$.
\end{lemma}
\begin{proof} Since polynomials only have finitely many zeros there is a $K' > 0$ such that $p_i(x) < q_j(x)$ for all $i,j$ and all $x \geq K'$. Write $q_j(x) = a_{0,j} + a_{1,j}x + a_{2,j}x^2 + \cdots + a_{N,j}x^N$ and $p_i(x) = b_{0,j} + b_{1,j}x + b_{2,j}x^2 + \cdots + b_{N,j}x^N$ for some $N$ larger than the degree of any of the four given polynomials, and set
$$
\xi_i = (b_{N,i}, b_{N-1,i}, \cdots , b_{0,i}) \in \mathbb Q^{N+1}
$$
and
$$
\eta_j = (a_{N,j}, a_{N-1,j}, \cdots , a_{0,j}) \in \mathbb Q^{N+1} .
$$
Since $p_i(x) < q_j(x)$ for all large $x$, we have that $\xi_i <_{lex} \eta_j$ for all $i,j$ with respect to the lexicographic order $<_{lex}$. Since $\mathbb Q^{N+1}$ is totally ordered in the lexicographic order there is an element $ c = (c_N,c_{N-1}, c_{N-2}, \cdots , c_0) \in \mathbb Q^{N+1}$ such that $\xi_i <_{lex} c <_{lex} \eta_j$ for all $i,j$. By changing $c_0$ by a small amount we can arrange that $c \notin \{\xi_1,\xi_2,\eta_1,\eta_2\}$. Then the polynomial
$$
h(x) = c_0 + c_1x + c_2 x^2 + \cdots + c_Nx^N
$$
will have the desired property.
\end{proof}

Consider the subset $\Gamma$ of $\left( \bigoplus_\mathbb Z H\right) \oplus G$ consisting of the elements $(\xi, g)  \in \left( \bigoplus_\mathbb Z H\right) \oplus G$ with the property that there is an $\epsilon > 0$ such that
\begin{equation}\label{06-09-21d}
\hat{L}(\xi)(x) = g(x) \ \ \forall x \in \pi^{-1}(]-\epsilon,\epsilon[) \ .
\end{equation}
$\Gamma$ is a subgroup of $\left(\bigoplus_\mathbb Z H\right) \oplus G$.

\begin{lemma}\label{05-09-21} The projection $\Gamma \to G$ is surjective.
\end{lemma}
\begin{proof} Let $g \in G$. By definition of $G$ there is an element $\xi  \in \bigoplus_{\mathbb Z}H$ and $k \in \mathbb N$ such that $g(x) = \hat{L}(\xi)$ on $\pi^{-1}(\left]-\frac{1}{2k},\frac{1}{2k}\right[$.
\end{proof}

Set
$$
\Gamma^+ = \left\{  (\xi, g)  \in \Gamma : \ g \in G^+ \backslash \{0\} \right\} \cup \{0\} \ .
$$
By combining Lemma \ref{05-09-21} and Lemma \ref{01-09-21} above with Lemma 3.1 and Lemma 3.2 in \cite{EHS} we conclude that $(\Gamma,\Gamma^+)$ is a dimension group.

 Given an element $h \in H$ we denote in the following by $h^{(0)}$ the element of $\bigoplus_{\mathbb Z} H$ defined by $(h^{(0)})_0 = h$ and $(h^{(0)})_n = 0$ when $n \neq 0$. Define $\Sigma : \bigoplus_\mathbb Z H \to H$ by
$$
\Sigma\left((h_n)_{n \in \mathbb Z}\right) = \sum_{n \in \mathbb Z} h_n \ .
$$

\begin{lemma}\label{01-09-21d} $(\Gamma,\Gamma^+)$ has large denominators; that is for all $x \in \Gamma^+$ and $m \in \mathbb N$ there is an element $y \in \Gamma^+$ and an $n \in \mathbb N$ such that $my \leq x \leq ny$. 
\end{lemma}
\begin{proof} 

Let $ x=(\xi,g) \in \Gamma^+\backslash \{0\}$ and $m \in \mathbb N$ be given. Then $\Sigma(\xi) \in H^+\backslash \{0\}$ and since $H$ has large denominators by \cite{Ni}, there is an element $b\in H^+$ such that $ mb< \Sigma(\xi)  < n b$ for some $n \in \mathbb N, \ n > m+2$. Since $L(\theta(\Sigma(\xi)))$ agrees with $g$ on $\pi^{-1}(0)$, there is a compact neighborhood $U$ of $0$ such that
$$
mL(\theta(b))(x) < g(x) < n L(\theta(b))(x)
$$
for all $x \in \pi^{-1}(U)$. There is also a $K \in \mathbb N$ such that $U \subseteq \ ]-K,K[$ and functions $f^\pm \in \mathbb Q[e^{-\pi}, 1-e^{-\pi}]$ such that $g(x) = f^-(x), \ x \leq -K$, and $g(x) = f^+(x),\ x \geq K$. It follows from (2) of Lemma \ref{03-09-21a} that there is an element $a \in \mathcal A(S,\pi)$ such that 
\begin{itemize}
\item[{}] $\frac{1}{n} g(x) < a(x) < \frac{1}{m} g(x) \ \ \forall x \in S$,
\item[{}] $a(x) = L(\theta(b))(x)$ for all $x \in \pi^{-1}(U)$,
\item[{}]  $a(x) = \frac{1}{ m+1}f^-(x)$ for $x \leq -K$, and
\item[{}] $a(x) = \frac{1}{m+1}f^+(x)$ for all $x \geq K$.
\end{itemize}
Choose $k \in \mathbb N$ so large that $[-\frac{1}{k},\frac{1}{k}] \subseteq U$ and note that 
$$
a(x) = L(\theta(b))(x)\psi^0_k\circ \pi(x) + a(x)\psi^+_k\circ\pi(x)  +a(x)\psi^-_k\circ \pi(x)  .
$$
Then the function
$$
a' := a - L(\theta(b))\psi^0_k\circ \pi - \frac{1}{m+1}f^+\psi^+_k\circ \pi - \frac{1}{m+1}f^-\psi^-_k\circ \pi
$$
is supported in 
$]-K,K[ \backslash \{0\}$. Let $\delta > 0$ be smaller than $\frac{1}{m}g(x) - a(x)$ and $a(x) - \frac{1}{n}g(x)$ for all $x \in \pi^{-1}([-K,K])$. By Property \ref{07-09-21} we can find $c \in G_{00}$ such that $\supp c \subseteq ]-K,K[ \backslash \{0\}$ and $\left|c(x) -a'(x)\right| < \delta$ for all $x \in S$. Then
$$
g' := c  +L(\theta(b))\psi^0_k\circ \pi + \frac{1}{m+1}f^+\psi^+_k\circ \pi + \frac{1}{m+1}f^-\psi^-_k\circ \pi \ \in \ G
$$
and $mg'(x) < g(x) < ng'(x)$ for all $x \in S$. It follows that $y = (b^{(0)},g) \in \Gamma^+$ has the desired property.
\end{proof}

Since $\hat{L} \circ \sigma = \alpha_0 \circ \hat{L}$ we can define $\alpha \in \Aut \Gamma$ by
$$
\alpha = \sigma \oplus \alpha_0  .
$$

\begin{lemma}\label{04-08-21x} The only order ideals $I$ in $\Gamma$ such that $\alpha(I) = I$ are $I = \{0\}$ and $I = \Gamma$.
\end{lemma}
\begin{proof} Recall that an order ideal $I$ in $\Gamma$ is a subgroup such that
 \begin{itemize}
\item[(a)] $I= I \cap \Gamma^+ -  I \cap \Gamma^+$, and
\item[(b)] when $ 0 \leq y \leq x$ in $\Gamma$ and $x \in I$, then $y \in I$.
\end{itemize}
Let $I$ be a non-zero order ideal such that $\alpha(I) = I$. Since $I \cap \Gamma^+ \neq \{0\}$ there is an element $g \in G^+ \backslash \{0\}$ and an element $\xi \in \bigoplus_{n \in \mathbb Z}H$ such that $(\xi,g) \in I$. Set $h = \Sigma(\xi)$. By definition of $G^+ \backslash \{0\}$ there are natural numbers $n, m, k, K \in \mathbb N$ such that the function
$$
g' = L(\theta(h))\psi^0_k \circ \pi + e^{n\pi}\psi^-_k\circ \pi  +  e^{-m\pi}\psi^+_k\circ \pi
$$
has the property that
$$
0 < g'(x) < K g(x) \  \ \ \forall x \in S   .
$$ 
It follows that $({h}^{(0)}, g') \in I \cap \Gamma^+$. Note that
$$
\alpha_0^l(g') =  e^{-l\pi} L(\theta(h))\psi^0_k \circ \pi +  e^{(n-l)\pi}\psi^-_k\circ \pi +  e^{-(m+l)\pi}\psi^+_k\circ \pi
$$
for all $l \in \mathbb Z$. Consider an arbitrary element $(\xi',f) \in \Gamma^+ \backslash \{0\}$. We can then find $l_1,l_2 \in \mathbb Z$ and $ M \in \mathbb N$ such that
$$
f(x) < M\left( \alpha_0^{l_1}(g')(x) + \alpha_0^{l_2}(g')(x)\right) \ \ \forall x \in S  .
$$
Since 
$$
\alpha^{l_1}((h^{(0)},g')) +  \alpha^{l_2}((h^{(0)},g')) \in I \cap \Gamma^+  ,
$$ 
it follows that $(\xi',f) \in I$ and we conclude therefore that $I = \Gamma$.
\end{proof}

\subsection{Some homomorphisms $\Gamma \to \mathbb R$}
Note that the constant function $1$ is in $G$ and that, with $v := (u^{(0)},1)$, $v \in \Gamma^+$.
Let $\beta \in \mathbb R$ and $\omega \in \pi^{-1}(\beta)$. As in Remark \ref{07-09-21f} we denote in the sequel by $\mathcal A_\mathbb R(S,\pi)$ the real Banach space consisting of the elements of $\mathcal A(S,\pi)$ that have a limit at infinity.

\begin{lemma}\label{01-09-21e} Let $f \in \mathcal A_\mathbb R(S,\pi)$ and let $\epsilon > 0$ be given. There is an element $g\in G$ such that $\sup_{x \in S} |f(x)- g(x)| \leq \epsilon $.
\end{lemma}
\begin{proof} An initial approximation gives us an element $f_1 \in \mathcal A(S,\pi)$ which is compactly supported and a real number $r \in \mathbb R$ such that 
$$
\sup_{x \in S} |f(x)- f_1(x) - r| \leq \frac{\epsilon}{2} \ .
$$
Let $q \in \mathbb Q$ such that $|q-r| < \frac{\epsilon}{6}$ and choose an element $h \in H$ such that $|\theta(h)(y) - f_1(y) -r| < \frac{\epsilon}{6}$ for all $y \in \pi^{-1}(0)$. There is a $k \in \mathbb N$ such that $\left|L(\theta(h))(x) -f_1(x) -r\right| < \frac{\epsilon}{6}$ for all $x \in \pi^{-1}([-\frac{1}{k},\frac{1}{k}])$. Since $f_1\psi^+_k\circ \pi + f_1\psi^-_k\circ \pi$ is compactly supported in $S \backslash \pi^{-1}(0)$ it follows from Property \ref{07-09-21} that there is an element $g' \in G_{00}$ such that
$$
\sup_{x \in S}\left| g'(x) - f_1(x)\psi^+_k\circ \pi(x) - f_1(x)\psi^-_k\circ \pi(x)\right| \leq \frac{\epsilon}{6} \ .
$$
Then 
$$
g = L(\theta(h))\psi^0_k \circ \pi + q\psi^+_k\circ \pi  + q\psi^-_k\circ \pi  + g' \in G
$$
is an element with the desired property.
\end{proof}

Let $\beta \in \mathbb R$. For each $\omega \in \pi^{-1}(\beta)$,  define $\omega_\beta : \Gamma \to \mathbb R$ by
$$
\omega_\beta(\xi,g) = g(\omega) \ .
$$ 
Then $\omega_\beta(\Gamma^+) \subseteq [0,\infty)$, $\omega_\beta(v) = 1$, and $\omega_\beta \circ \alpha = e^{-\beta} \omega_\beta$.

\begin{lemma}\label{27-08-21x}
Let $\phi : \Gamma \to \mathbb R$ be a positive homomorphism with the properties that $\phi(v) =1$ and $\phi \circ \alpha = s \phi$ for some $s > 0$. Set $\beta = -\log s$. There is an element $\omega \in \pi^{-1}(\beta)$ such that $\phi = \omega_\beta$.
\end{lemma}
\begin{proof} The projection $p : \Gamma \to G$ is surjective by Lemma \ref{05-09-21}. Assume that $(\xi,g) \in \Gamma$ and $p(\xi,g) = g = 0$. Since $(\Gamma,\Gamma^+)$ has large denominators by Lemma \ref{01-09-21d} there is for each $n \in \mathbb N$ an element $(\xi_n,g_n) \in \Gamma^+$ and a natural number $k_n$ such that $n(\xi_n,g_n) \leq v \leq k_n (\xi_n,g_n)$. Then $\pm (\xi,g) \leq (\xi_n,g_n)$ in $\Gamma$ and hence $\pm \phi(\xi,g) \leq \phi(\xi_n,g_n) \leq \frac{1}{n}$. It follows that $\phi(\xi,g) = 0$ and we conclude that there is a homomorphism $\phi' : G \to \mathbb R$ such that $\phi' \circ p = \phi$. Let $g \in G$ and assume that $g(x) \geq 0$ for all $x \in S$, and let $n \in \mathbb N$. There is $h_n \in H^+$ such that $0 < \theta(h_n) < \frac{1}{n}$, and then also a natural number $k \in \mathbb N$ such that $0 <L(\theta(h_n))(x)\psi^0_k \circ \pi(x) + \frac{1}{2n}\psi^+_k(\pi(x)) + \frac{1}{2n}\psi^-_k(\pi(x)) <\frac{1}{n}$ for all $x \in S$. Then 
$$
g'_n := L(\theta(h_n))\psi^0_k\circ \pi + \frac{1}{2n}\psi^-_k\circ \pi + \frac{1}{2n}\psi^+_k\circ \pi\in G^+
$$ 
and $0 \leq n(h_n^{(0)},g_n')\leq v$ in $\Gamma$. Hence $0 \leq \phi'(g_n') \leq \frac{1}{n}$. Let $\xi \in \bigoplus_{n\in \mathbb Z} H$ be an element such that $(\xi,g) \in \Gamma$. Then $(h_n^{(0)} +\xi,g_n'+g) \in \Gamma^+$ and hence $0 \leq \phi'(g_n'+g) \leq \phi'(g) + \frac{1}{n}$. Letting $n$ tend to infinity we find that $\phi'(g) \geq 0$, proving that $\phi'$ is positive on $G$. Let $g \in G$ and $n,m \in \mathbb N$ satisfy $|g(x)| < \frac{n}{m}$ for all $x \in S$. Then $-n <mg(x) < n$ for all $x \in S$ and since $\phi'(1) = 1$ this leads to the conclusion that $|\phi'(g)| \leq \frac{n}{m}$. Combined with Lemma \ref{01-09-21e} it follows from the last estimate that $\phi'$ extends by continuity to a linear map $\phi' : \mathcal A_\mathbb R(S,\pi)  \to \mathbb R$ such that $\left|\phi'(f)\right| \leq \sup_{x \in S} |f(x)|$. Using a Hahn-Banach theorem we extend $\phi'$ in a norm-preserving way to the space of all continuous real-valued functions on $S$ with a limit at infinity. Since $\phi'(1) = 1$ the extension is positive. It follows that there is a bounded Borel measure $m$ on $S$ such that
$$
\phi'(f) = \int_S f(x) \ \mathrm{d} m
$$
for all $f \in \mathcal A_{0}(S,\pi)$, where $\mathcal A_0(S,\pi)$ denotes the space of elements in $\mathcal A_\mathbb R(S,\pi)$ that vanish at infinity. 
Let $C_c(\mathbb R)$ denote the set of continuous real-valued compactly supported functions on $\mathbb R$ and note that $C_c(\mathbb R)$ is mapped into $\mathcal A_0(S,\pi)$ by the formula $F \mapsto F \circ \pi$. Since $\phi \circ \alpha = s\phi$ by assumption it follows that the measure $m \circ \pi^{-1}$ on $\mathbb R$ satisfies 
$$
 \int_\mathbb R e^{-t}F(t) \ \mathrm{d}m \circ \pi^{-1}(t)  =  s\int_\mathbb R F(t) \ \mathrm{d}m \circ \pi^{-1}(t)  \ \ \forall F \in C_c(\mathbb R)  .
 $$
It follows that $m \circ \pi^{-1}$ is concentrated at the point $\beta = -\log s$ and hence that $m$ is concentrated on $\pi^{-1}(\beta)$. We can therefore define a linear functional $\phi'' : \Aff \pi^{-1}(\beta) \to \mathbb R$ by
$$
\phi''(f) = \phi'(\hat{f}) = \int_S \hat{f}(x) \ \mathrm{d} m(x) \ ,
$$
where $\hat{f} \in \mathcal A_0(S,\pi)$ is any element with $\hat{f}|_{\pi^{-1}(\beta)} = f$, which exists by (1) in Lemma \ref{03-09-21a}. If $f \geq 0$ it follows from (2) of Lemma \ref{03-09-21a} that $\hat{f}$ can be chosen such that $\hat{f} \geq - \epsilon$ for any $\epsilon > 0$ and we see therefore that $\phi''$ is a positive linear functional. Since every state of $\Aff \pi^{-1}(\beta)$ is given by evaluation at a point in $\pi^{-1}(\beta)$ it follows in this way that there is an $\omega \in \pi^{-1}(\beta)$ and a number $\lambda \geq 0$ such that 
\begin{equation}\label{01-09-21h}
\phi'(g) = \lambda g(\omega)
\end{equation}
for all $g  \in \mathcal A_0(S,\pi)$. In particular, this conclusion holds for all $g \in G \cap \mathcal A_0(S,\pi)$. A general element $f \in G$ can be write as a sum
$$
f = f_- + f_0 + f_+ ,
$$
where $f_\pm, f_0 \in G$, $f_0$ has compact support and there are natural numbers $n_{\pm} \in \mathbb N$ such that $e^{n_-\pi}f_- \in \mathcal A_0(S,\pi)$ and $e^{-n_+ \pi}f_+ \in \mathcal A_0(S,\pi)$. Then $\phi'(f_0) = \lambda f_0(\omega)$,
\begin{align*}
&\phi'(f_-) = \phi'(\alpha^{n_-}(e^{n_-\pi}f_-)) = s^{n_-}\phi'(e^{n_-\pi}f_-) \\
&=  s^{n_-}\lambda e^{n_-\pi(\omega)}f_-(\omega) = \lambda f_-(\omega),
\end{align*}
and similarly, $\phi'(f_+) = \lambda f_+(\omega)$.
It follows that $\phi(f) = \lambda f(\omega)$. Inserting $f=1$ we find that $\lambda = 1$ and the proof is complete.
\end{proof}

\subsection{Application of the Pimsner-Voiculescu exact sequence}

Let $B$ be a stable AF algebra with $(K_0(B),K_0(B)^+) = (\Gamma, \Gamma^+)$ and let $\gamma$ be an automorphism of $B$ such that $\gamma_* = \alpha$; see \cite{E1}. 

\begin{poem}\label{07-09-21c} By Lemma \ref{08-08-21a} we can arrange that $\gamma$ has the following additional properties:
\begin{itemize}
\item[(A)] The restriction map 
$\mu \ \mapsto \ \mu|_B$
is a bijection from traces $\mu $ on $B \rtimes_{\gamma} \mathbb Z$ onto the $\gamma$-invariant traces on $B$, and
\item[(B)] $B \rtimes_{\gamma} \mathbb Z$ is $\mathcal Z$-stable; that is $(B \rtimes_{\gamma} \mathbb Z)\otimes \mathcal Z \simeq B \rtimes_{\gamma} \mathbb Z$ where $\mathcal Z$ denotes the Jiang-Su algebra, \cite{JS}.
\end{itemize} 
\end{poem}
Set
$$
C = B\rtimes_\gamma \mathbb Z  \ .
$$
It follows from Lemma \ref{04-08-21x} and \cite{E1} that $B$ is $\gamma$-simple and hence from \cite{E2} (see also\cite{Ki1}) that $C$ is simple. 
It follows from the Pimsner-Voiculescu exact sequence, \cite{PV}, that we can identify $K_0(C)$, as a group, with the quotient
$$
\Gamma/(\id - \alpha)(\Gamma)  ,
$$
in such a way that the map $\iota_* : K_0(B) \to K_0(C)$ induced by the inclusion $\iota : B \to C$ becomes the quotient map
$$
q : \Gamma \to  \Gamma/(\id - \alpha)(\Gamma) .
$$
Define $S_0 : \Gamma \to H$ such that 
$$
S_0(\xi, g) = \Sigma(\xi ) .
$$

\begin{lemma}\label{01-09-21k} $\ker S_0 = (\id- \alpha)(\Gamma)$.
\end{lemma}
\begin{proof} Since $ (\id - \alpha)(\Gamma) = (\id -\sigma) \oplus (\id - \alpha_0)$ and $\Sigma \circ (\id - \sigma) = 0$, we find that $ (\id - \alpha)(\Gamma) \subseteq \ker S_0$. Let $(\xi,g) \in \Gamma$, and assume that $S_0(\xi,g) =\Sigma(\xi) = 0$. By Lemma 4.6 of \cite{Th3} there is an element $\xi' \in \bigoplus_\mathbb Z H$ such that $(\id - \sigma)(\xi') = \xi$. By the definition of $\Gamma$ there is $\epsilon > 0$ such that $\hat{L}(\xi)$ and $g$ agree on $\pi^{-1}(]-\epsilon,\epsilon[)$, and so when $k \geq \epsilon^{-1}$ we have 
$$
g = \hat{L}(\xi)\psi^0_k\circ \pi + h^-\psi^-_k\circ \pi + h^+\psi^+_k\circ \pi + g_0
$$ 
for some $h^\pm \in \mathbb Q[e^{-\pi},1-e^{-\pi}]$ and some $g_0 \in G_{00}$. By the definition of $\mathbb Q[e^{-\pi},1-e^{-\pi}]$ there are elements $f^\pm \in \mathbb Q[e^{-\pi},1-e^{-\pi}]$ such that $h^\pm = (\id -\alpha_0)(f^\pm)$ and by \eqref{07-09-21e} there is an element $g' \in G_{00}$ such that $g_0 = (\id -\alpha_0)(g')$. Define
$$
g'' :=  \hat{L}(\xi')\psi^0_k\circ \pi + f^-\psi^-_k\circ \pi + f^+\psi^+_k\circ \pi + g' \ \in \ G .
$$
Since 
$$
 (\id - \alpha_0)(\hat{L}(\xi')\psi^0_k\circ \pi) = \hat{L}((\id - \sigma)(\xi'))\psi^0_k\circ \pi = \hat{L}(\xi)\psi^0_k\circ \pi , 
 $$
 it follows that $g = (\id - \alpha_0)(g'')$ and hence that $(\xi,g) = (\id - \alpha)( \xi',g'')$.
 \end{proof}

It follows from Lemma \ref{01-09-21k} that $S_0$ induces an isomorphism
$$
S : K_0(C) = \Gamma/(\id - \alpha)(\Gamma)  \to H \ 
$$
such that $S \circ q = S_0$.

\begin{lemma}\label{02-09-21} $S(K_0(C)^+) = H^+$.
\end{lemma}
\begin{proof} Let $h \in H^+ \backslash \{0\}$. There is then a $k$ so big that $L(\theta(h))(x) > 0$ for all $x \in \pi^{-1}([-\frac{1}{k},\frac{1}{k}])$. Define
$$
g :=  L(\theta(h))\psi^0_k\circ \pi + \psi^-_k \circ \pi + \psi^+_k \circ \pi \in G_k \ .
$$
Then $(h^{(0)},g) \in \Gamma^+, \ q((h^{(0)},g)) \in K_0(C)^+$, and $S( q((h^{(0)},g))) = h$. Hence,
$S(K_0(C)^+) \supseteq H^+$. Consider an element $x \in K_0(C)^+ \backslash \{0\}$ and write $x = q(\xi,g)$ for some $(\xi,g) \in \Gamma$. Let $\omega \in \pi^{-1}(0)$. Since $\omega_0 \circ \alpha = \omega_0$, there is a $\gamma$-invariant trace $\tau_\omega$ on $B$ such that ${\tau_\omega}_* = \omega_0$; see Lemma 3.5 in \cite{Th3}. Denote by $P : C \to B$ the canonical conditional expectation and note that $\tau_\omega \circ P$ is a trace on $C$. Since $x \in K_0(C)^+ \backslash \{0\}$ and $C$ is simple it follows that
$$
{(\tau_\omega \circ P)}_*(x) > 0 .
$$
Since ${(\tau_\omega \circ P)}_*(x) = \Sigma(\xi)(\omega)$, and $\omega \in \pi^{-1}(0)$ was arbitrary, it follows that $S(x) = \Sigma(\xi) \in H^+ \backslash \{0\}$. Hence, $S(K_0(C)^+) \subseteq H^+$.
\end{proof}

\begin{lemma}\label{02-09-21c} $K_1(C) = 0$.
\end{lemma}
\begin{proof} To establish this from the Pimsner-Voiculescu exact sequence, \cite{PV}, we must show that $\id -\alpha$ is injective. Let $(\xi,g) \in \Gamma$ and assume that $\alpha(\xi,g) = (\xi,g)$. Then $\sigma(\xi) = \xi$, implying that $\xi = 0$ and hence that $g|_{\pi^{-1}(0)} =0$. Since $(1-e^{-\pi(x)})g(x) = 0$ for all $x\in S$, it follows that $g =0$.
\end{proof}

Let $e \in B$ be a projection such that $[e] = v$ in $K_0(B) = \Gamma$. Since $eCe$ is stably isomorphic to $C$ by \cite{B} it follows that $(K_0(eCe),K_0(eCe)^+) = (K_0(C), K_0(C)^+)$.

\subsection{Completing the proof of Theorem \ref{26-08-21} via classification theory}

Let $(S,\pi)$ and $A$ be as in Theorem \ref{26-08-21}. With $H = K_0(A)$ and the assumed identification of the tracial state space $T(A)$ of $A$ with $\pi^{-1}(0)$ we get the homomorphism $\theta : H \to \Aff \pi^{-1}(0)$ from the canonical map $K_0(A) \to \Aff T(A)$. It follows from Theorem 4.11 in \cite{GH} that $\theta(K_0(A))$ is dense in $\Aff \pi^{-1}(0)$ and that
$$
K_0(A)^+ = \left\{ h \in K_0(A): \ \theta(h)(x) > 0 \ \forall x \in \pi^{-1}(0)\right\} \cup \{0\} .
$$
We can therefore apply the preceding with $H = K_0(A), \ H^+ = K_0(A)^+$ and $u = [1]$.

Let $\tau$ be a trace state on $eCe$. Then $\tau_* \circ S^{-1} : H \to \mathbb R$ is a positive homomorphism such that 
$\tau_* \circ S^{-1}(u) = \tau_*(q(v)) = \tau(e) = 1$, and there is therefore a unique trace state 
$\tau'$ on $A$ such that 
$$
{\tau'}_* =  \tau_* \circ S^{-1}
$$ 
on $K_0(A) = H$.

\begin{lemma}\label{02-09-21b} The map $\tau \to \tau'$ is an affine homeomorphism from $T(eCe)$ onto $T(A)$.
\end{lemma}
\begin{proof} The map is clearly affine. To show that it is continuous, assume that $\{\tau_n\}$ is a convergent sequence in $T(eCe)$ and let $\tau = \lim_{n \to \infty} \tau_n$. Then $\lim_{n \to \infty} {\tau_n}_* \circ S^{-1}(h) = \tau_*\circ S^{-1}(h)$ for all $h \in H$. Since $A$ is AF this implies that $\lim_{n \to \infty} \tau'_n = \tau'$ in $T(A)$. To see that the map is surjective, let $\tau \in T(A)$. Then $\tau_*: H \to \mathbb R$ is given by evaluation at a point $\omega \in \pi^{-1}(0)$, and $\tau_1 = \tau_{\omega}\circ P$ is a trace state on $eCe$ such that $\tau_1' = \tau$. To see that the map is also injective, consider $\tau_1,\tau_2 \in T(eCe)$. If $\tau_1' = \tau_2'$, it follows that ${\tau_1}_* = {\tau_2}_*$. Since ${\tau_1}_*\circ \iota_* = {\tau_2}_*\circ \iota_*$ and $B$ is AF it follows that $\tau_1|_B = \tau_2|_B$. Thanks to (A) from Additional properties \ref{07-09-21c} this implies that $\tau_1 = \tau_2$.
\end{proof}

\begin{lemma}\label{02-09-21k} $eCe$ is $*$-isomorphic to $A$.
\end{lemma}
\begin{proof}  Since $A$ is AF the $K_1$ group of $A$ is trivial, and by Lemma \ref{02-09-21c} the same is true for $eCe$ since $eCe$ is stably isomorphic to $C$. The affine homeomorphism $\tau \to \tau'$ of Lemma  \ref{02-09-21b} is compatible with the isomorphism of ordered groups $S : K_0(eCe) \to K_0(A)$ from Lemma \ref{02-09-21} in the sense that ${\tau'}_* \circ S = \tau_*$, resulting in an isomorphism from the Elliott invariant of $eCe$ onto that of $A$. Both algebras, $A$ and $eCe$, are separable, simple, unital, nuclear and in the UCT class. It is well known that all infinite-dimensional unital simple AF algebras are approximately divisible and hence $\mathcal Z$-absorbing by Theorem 2.3 of \cite{TW}; in particular, $A$ is $\mathcal Z$-absorbing. Since $C$ is $\mathcal Z$-absorbing thanks to (B) in Additional properties \ref{07-09-21c}, it follows from Corollary 3.2 of \cite{TW} that $eCe$ is $\mathcal Z$-absorbing. Therefore $eCe$ is isomorphic to $A$ by Corollary D of \cite{CETWW}, which in turn is based on \cite{GLN1}, \cite{GLN2}, \cite{EGLN} and \cite{TWW}. (In the case where $A$ is a UHF algebra there is an alternative route through the literature to the same effect. See Remark 4.12 in \cite{Th3}.)
\end{proof}

We consider the dual action on $C = B  \rtimes_\gamma \mathbb Z$ as a $2\pi$-periodic flow and we denote by $\theta$ the restriction of this flow to $eCe$.

\begin{lemma}\label{02-09-21d} The KMS bundle $(S^\theta,\pi^\theta)$ of $\theta$ is isomorphic to $(S,\pi)$.
\end{lemma}
\begin{proof} Let $(\omega,\beta) \in S^\theta$. By Corollary \ref{06-08-21a}, $(\hat{\omega}|_B)_*$ is a positive homomorphism $\Gamma \to \mathbb R$ such that $(\hat{\omega}|_B)_*(v) =1$ and $(\hat{\omega}|_B)_* \circ \alpha = e^{-\beta}(\hat{\omega}|_B)_*$. By Lemma \ref{27-08-21x}, there is $\mu \in \pi^{-1}(\beta)$ such that $(\hat{\omega}|_B)_*(\xi,g) = g(\mu)$ for all $(\xi,g) \in \Gamma$. $\mu$ is unique since $G$ separates the points of $S$ by Lemma \ref{01-09-21e}. We define $\Phi : S^\theta \to S$ by $\Phi(\omega,\beta) = \mu$. By combining Lemma \ref{27-08-21x} and Corollary \ref{06-08-21a} we conclude that $\Phi$ restricts to an affine bijection from ${\pi^\theta}^{-1}(\beta)$ onto $\pi^{-1}(\beta)$ for every $\beta \in \mathbb R$. It follows in particular that $\Phi$ is surjective. If $(\omega_i,\beta_i) \in S^\theta, \ i =1,2$, are such that $\Phi((\omega_1,\beta_1)) = \Phi((\omega_2,\beta_2))$, it follows that $\beta_1 = \pi\left(\Phi((\omega_1,\beta_1))\right) = \pi\left(\Phi((\omega_2,\beta_2))\right) = \beta_2$ and hence that $(\omega_1,\beta_1) = (\omega_2,\beta_2)$. Thus, $\Phi$ is a bijection. Since $\pi \circ \Phi = \pi^\theta$, and $\pi$ and $\pi^\theta$ are both proper maps, it suffices to show that $\Phi^{-1}$ is continuous. Let therefore $\{\omega^n\}$ be a sequence in $S$ such that $\lim_{n \to \infty} \omega^n = \omega$ in $S$. Set $\beta_n = \pi(\omega^n)$ and note that $\lim_{n \to \infty} \beta_n = \beta$, where $\beta = \pi(\omega)$. It follows that $\lim_{n \to \infty} \omega^n_{\beta_n}(x) = \omega_\beta(x)$ for all $x \in \Gamma$. Let $\tau^n$ and $\tau$ be the traces on $B$ determined by the conditions that ${\tau^n}_* = \omega^n_{\beta_n}$ and $\tau_* = \omega_\beta$. Then $\Phi^{-1}(\omega^n) = (\tau^n\circ P|_{eCe},\beta_n)$ and $\Phi^{-1}(\omega) = (\tau \circ P|_{eCe},\beta)$. It suffices therefore to show that $\lim_{n \to \infty} \tau^n\circ P(exe) = \tau \circ P(exe)$ for all $x \in C$. Since $\tau^n\circ P(e) = \tau \circ P(e) =1$, it suffices to check for $x$ in a dense subset of $C$. If $w$ is the canonical unitary in the multiplier algebra of $C$ coming from the construction of $C$ as a crossed product, it suffices to show that $\lim_{n \to \infty} \tau^n\circ P(ebw^ke) = \tau \circ P(ebw^ke)$ for all $k \in \mathbb Z$ and all $b \in B$. Since $P(ebw^ke) = 0$ when $k \neq 0$ it suffices to consider the case $k = 0$; that is, it suffices to show that
$\lim_{n \to \infty} \tau^n (ebe) = \tau(ebe)$.
By approximating $ebe$ by a linear combination of projections from $eBe$ it suffices to show that $\lim_{n \to \infty} \tau^n(p) = \tau(p)$
when $p$ is a projection in $eBe$. This holds because 
$$
\lim_{n \to \infty}\tau^n(p)  = \lim_{n \to \infty} \omega^n_{\beta_n}([p]) = \omega_\beta([p]) = \tau(p).
$$

\end{proof}

The proof of Theorem \ref{26-08-21} is complete.


\begin{thebibliography}{WWWWW} 






\bibitem[BEH]{BEH} O. Bratteli, G. A. Elliott and R.H. Herman, {\em On the possible temperatures of a dynamical system}, Comm. Math. Phys. {\bf 74} (1980), 281--295.

\bibitem[BEK1]{BEK1} O. Bratteli, G. A. Elliott and A. Kishimoto, {\em The temperature state space of a dynamical system I}, J. Yokohama Univ. {\bf 28} (1980), 125--167. 

\bibitem[BEK2]{BEK2} O. Bratteli, G. A. Elliott and A. Kishimoto, {\em The temperature state space of a dynamical system II }, Ann. of Math. {\bf 123} (1986), 205--263.

\bibitem[BR]{BR} O. Bratteli and D.W. Robinson, {\em Operator Algebras and Quantum Statistical Mechanics I + II}, Texts and Monographs in Physics, Springer Verlag, New York, Heidelberg, Berlin, 1979 and 1981.

  
\bibitem[B]{B} L. G. Brown, {\em Stable isomorphism of hereditary subalgebras of $C^*$-algebras}, Pacific J. Math. {\bf 71} (1977), 335--348.  
 
 

  
    
     
\bibitem[CETWW]{CETWW} J. Castillejos, S. Evington, A. Tikuisis, 
S. White and W. Winter, {\em Nuclear dimension of simple $C^*$-algebras}, arXiv:1901.05853v3, Invent. Math., to appear. 












\bibitem[EHS]{EHS} E.G. Effros, D.E. Handelman and C.-L. Shen, {\em Dimension groups and their affine representations}, Amer. J. Math. {\bf 102} (1980), 385--407.


\bibitem[E1]{E1} G. A. Elliott, {\em On the classification of inductive limits of sequences of semisimple finite-dimensional algebras}, J. Algebra {\bf 38} (1976), 29--44.

\bibitem[E2]{E2} G. A. Elliott, {\em Some simple $C^*$-algebras constructed as crossed products with discrete outer
automorphism groups}, Publ. RIMS, Kyoto Univ. {\bf 16} (1980), 299--311.




\bibitem[EGLN]{EGLN} G. A. Elliott, G. Gong, H. Lin and Z. Niu, {\em On the classification of simple amenable $C^*$-algebras with finite decomposition rank, II,} preprint. arXiv:1507.03437.


\bibitem[EST]{EST}  G. A. Elliott, Y. Sato and K. Thomsen, {\em In preparation}.





\bibitem[GLN1]{GLN1} G. Gong, H. Lin and Z. Niu, {\em A classification of finite simple amenable Z-stable $C^*$-algebras, I: $C^*$-algebras with generalized tracial rank one}, C.R. Math. Acad. Sci. Soc. R. Can. {\bf 42} (2020), 63--450.



\bibitem[GLN2]{GLN2} G. Gong, H. Lin and Z. Niu, {\em A classification of finite simple amenable Z-stable $C^*$-algebras, II: $C^*$-algebras with rationalized generalized tracial rank one}, C.R. Math. Acad. Sci. Soc. R. Can. {\bf 42} (2020), 451--539.



 
 

\bibitem[GH]{GH} K.R. Goodearl and D.E. Handelman, {\em Metric Completions of Partially Ordered
Abelian Groups}, Indiana Univ. Math. J. {\bf 29} (1980), 861--895.
 



\bibitem[JS]{JS} X. Jiang and H. Su, {\em On a simple unital projectionless $C^*$-algebra}, Amer. J. Math. {\bf 121} (2000), 359--413.

\bibitem[Ki1]{Ki1} A. Kishimoto, {\em Outer Automorphisms and Reduced Crossed Products
of Simple $C^*$-Algebras}, Comm. Math. Phys. {\bf 81} (1981), 429--435.






\bibitem[LN]{LN} M. Laca and S. Neshveyev, {\em KMS states of quasi-free dynamics on Pimsner algebras}, J. Funct. Anal. {\bf 211} (2004), 457--482.

  
  
  
 


 
 
    
    
    
    


  

\bibitem[MS1]{MS1} H. Matui and Y. Sato, {\em Decomposition rank of UHF-absorbing $C^*$-algebras}, Duke Math. J. {\bf 163} (2014), 2687--2708.   
  
\bibitem[MS2]{MS2} H. Matui and Y. Sato, {\em Z-stability of crossed products by strongly outer actions}, Comm.Math. Phys. {\bf 314}(2012), 193--228.
 

\bibitem[Ni]{Ni} V. Nistor, {\em On the homotopy groups of the automorphism group of AF-C*-algebras}, J. Operator Theory {\bf 19} (1988), 319--340. 
 
  


\bibitem[PV]{PV} M. Pimsner and D. Voiculescu, {\em Exact sequences for $K$-groups and Ext-groups of certain cross-products of $C^*$-algebras}, J. Oper. Th. {\bf 4} (1980), 93--118.
  
 
 






 







\bibitem[Sa]{Sa} Y. Sato, {\em The Rohlin property for automorphisms of the Jiang--Su algebra}, J. Funct. Analysis {\bf 259} (2010) 453--476.



 
\bibitem[Th1]{Th1} K. Thomsen, {\em KMS weights on graph $C^*$-algebras}, Adv. Math. {\bf 309} (2017) 334--391.



\bibitem[Th2]{Th2} K. Thomsen, {\em Phase transition in the CAR algebra}, Adv. Math. {\bf 372} (2020); arXiv:1810.01828 


\bibitem[Th3]{Th3} K. Thomsen, {\em The possible temperatures for flows on a simple AF algebra}, Comm. Math. Phys. {\bf 386} (2021), 1489--1518.




\bibitem[TWW]{TWW} A. Tikuisis, S. White and W. Winter, {\em Quasidiagonality of nuclear $C^*$-algebras}, Ann. of Math. {\bf 185} (2017), 229--284.





 


\bibitem[TW]{TW} A. S. Toms and W. Winter, {\em Strongly self-absorbing $C^*$-algebras}, Trans. Amer. Math. Soc. {\bf 358} (2007), 3999--4029.




























\end{thebibliography}
\end{document}